\documentclass[11pt]{amsart}
\usepackage{amssymb}

\usepackage[all]{xy}
\theoremstyle{plain}
\newtheorem{theorem}{Theorem}[section]
\newtheorem{lemma}[theorem]{Lemma}
\newtheorem{proposition}[theorem]{Proposition} 
\newtheorem{corollary}[theorem]{Corollary}
\newtheorem{conjecture}[theorem]{Conjecture}

\theoremstyle{remark}
\newtheorem{remark}[theorem]{Remark}

\theoremstyle{definition}
\newtheorem{definition}[theorem]{Definition}

\newtheorem{example}[theorem]{Example}
\newcommand{\cA}{\mathcal{A}}

\newcommand{\cL}{\mathcal{L}}

\newcommand{\cO}{\mathcal{O}}

\newcommand{\C}{\mathbf{C}}
\newcommand{\Q}{\mathbf{Q}}
\newcommand{\Z}{\mathbf{Z}}
\newcommand{\Ps}{\mathbf{P}}
\newcommand{\bw}{\mathbf{w}}

\DeclareMathOperator{\lcm}{lcm}

\DeclareMathOperator{\rank}{rank}

\DeclareMathOperator{\prim}{prim}
\DeclareMathOperator{\Pic}{Pic}
\DeclareMathOperator{\Div}{Div}

\DeclareMathOperator{\Gr}{Gr}

\DeclareMathOperator{\ord}{ord}

\DeclareMathOperator{\tor}{tor}
\DeclareMathOperator{\Jac}{Jac}
\DeclareMathOperator{\NS}{NS}
\DeclareMathOperator{\Gal}{Gal}
\numberwithin{equation}{section}

\title[Toric decompositions]{Mordell-Weil lattices and toric decompositions of plane curves}

\author[R.~Kloosterman]{Remke Kloosterman}
\address{Institut f\"ur Mathematik, Humboldt-Universit\"at zu Berlin,
Unter den Linden 6, D-10099 Berlin, Germany} 
\email{klooster@math.hu-berlin.de}
\date{\today}
\thanks{The author would also like to thank  Orsola Tommasi for several comments on a previous version of this paper.}
\begin{document}

\begin{abstract}
We extend results of Cogolludo-Agustin and Libgober 
relating the Alexander polynomial of a plane curve $C$ with the Mordell--Weil rank of certain isotrivial families of jacobians over $\mathbf{P}^2$ of discriminant $C$.

In the second part we introduce a height pairing on the $(2,3,6)$ quasi-toric decompositions of a plane curve. We use this pairing and the results in the first part of the paper to construct a pair of degree 12 curves with 30 cusps and Alexander polynomial $t^2-t+1$, but with distinct height pairing. We use the height pairing to show that these curves from a Zariski pair.
\end{abstract}

\maketitle
\section{Introduction}\label{secIntro}
Let $C=V(g(z_0,z_1,z_2))\subset \Ps^2$ be a (reduced) plane curve. If $C$ is cuspidal and its degree is divisible by 6 then 
Cogolludo--Agustin and Libgober \cite{CogLib} expressed the Mordell-Weil rank of the elliptic threefold $X$ given by $y^2=x^3+g(s,t,1)$ in terms of the degree of the Alexander polynomial $\Delta_C$ of $C$. Their result relies on the following observations:
\begin{enumerate}
\item The elliptic threefold is birational to the quotient of $E\times S$ by an action of $\Z/6\Z$, where $E$ is the elliptic curve  $y^2=x^3+1$, and $S$ is the smooth projective model of  the surface $r^6=g(s,t,1)$.
\item A non-torsion element of the Mordell-Weil group  of $X$ yields a non-con\-stant morphism $S\to E$.
\item \label{isopower} The Albanese variety of $S$ is isogeneous to a power of $E$.
\item $h^1(S)$ equals twice the order of vanishing of $\Delta_C$ at a primitive sixth root of unity.
\end{enumerate}
Libgober  extended these results in the subsequent papers  \cite{LibIso,LibAlb}. In \cite{LibIso} he studied the Mordell-Weil rank of the Jacobian of $x^p+y^q+g(s,t,1)$ (considered as a curve over $\C(s,t)$), where $C={V(g(z_0,z_1,z_2))}$ is a plane curve which certain  prescribed types of singularities. In \cite{LibAlb} he dealt with certain two-dimensional families of  isotrivial abelian varieties where the discriminant of the family has prescribed singularities.

In this paper we consider  the Jacobian of the minimal smooth projective model of the curve $H:f(x,y)+g(s,t,1)$ over $\C(s,t)$, where $g$ is a weighted homogeneous polynomial in $x$ and $y$.
In our approach we use various Thom--Sebastiani type results to relate the vanishing order of the Alexander polynomial of $C$ to the Mordell--Weil rank of the Jacobian of $H$. 
Cogolludo--Libgober and Libgober needed to impose strong conditions on the type of singularities of $C$ in order to prove the property (\ref{isopower}). The advantage using Thom--Sebastiani type results is that we can replace these local constraints by a global, much weaker,  assumption on $C$.
However, the generic fiber of our families are always Jacobians of  plane curves. 
In the final section we will discuss a result where our approach is not optimal, i.e., we prove a result where the use of Albanese varieties yields a stronger statement than the Thom-Sebastiani type results.

To formulate our main result, we first have to recall Libgober's approach to calculate the Alexander polynomial of a plane curve \cite{LibPlane}. For a rational number $\beta$ in the interval $(0,1)$ he studied the ideals of regular functions whose constant of quasiadjunction is at least $\beta$ at every singular point of $C$. For practical convenience we denote the ideal of such functions by $I^{(1-\beta)}$. 
The ideal $I^{(\alpha)}$ is saturated and defines a zero-dimensional scheme of length $l_{\alpha}$. Let $d$ be the degree of $C$.  
Let $h_{\alpha}(k)$ be the Hilbert function  of $I^{(\alpha)}$. 
If $d \alpha$ is an integer, set 
\[ \delta_{\alpha}=l_{\alpha}-h_{\alpha}(\alpha d-3).\]
Denote  $\zeta(\alpha)=\exp(2\pi i\alpha)$.

Let $f(x_0,x_1)$ be a reduced weighted homogeneous polynomial.
For a rational number $\alpha$ set $\nu(\alpha)$ 
to be the multiplicity of $\alpha$ in the Steenbrink spectrum of $f$. Our main result is the following:

\begin{theorem}\label{thmMain} Let $f\in \C[x_0,x_1]$ be a weighted homogeneous polynomial with integral weights and of weighted degree $e$  and let $g\in \C[y_0,y_1,y_2]$ be a squarefree homogeneous polynomial of degree $d$.  Assume 
\[ e\mid d \mbox{ and }\sum_{0\leq \alpha<1} \nu(\alpha)\delta_{\alpha}=0.\]
Then the Mordell-Weil rank of the group of $\C(s,t)$-valued points of the Jacobian of the general fiber of $H:f(x,y)+g(s,t,1)$ equals
\[ \sum_{0<\alpha<1} (\nu(\alpha)+\nu(\alpha-1)) \ord_{t=\zeta(\alpha)}\Delta_C(t)\]
\end{theorem}
In the case that $d$ is even, $C$ is a curve with only ADE singularities and $f=x_1^2+x_2^e$ one easily shows that
\[ \sum_{0<\alpha<1} \nu(\alpha)\delta_{\alpha}=0\]
always holds.
Moreover, the spectrum of $f$ consists of $\frac{-1}{2}+\frac{i}{e}$ for $i=1,\dots,e-1$ and each number occurs with multiplicity one. In particular, we obtain

\begin{corollary}\label{corMain} Let  $g\in \C[y_0,y_1,y_2]$ be a homogeneous polynomial of even degree $d$, such that $C=V(f)$ is a curve with only ADE singularities.  Let $e$ be a divisor of $d$.
Then the Mordell-Weil rank of the group of $\C(s,t)$-valued points on the Jacobian of the general fiber of $H:y^2=x^e+g(s,t,1)$ equals
\[ 2 \sum_{i=1}^{\lfloor\frac{e-2}{2}\rfloor} \ord_{t=\zeta({1/2+i/e})} \Delta_C(t).\]
\end{corollary}
In particular, for $e=3,4$ we recover the main result of \cite{CogLib}. 

The structure of the proof of the main theorem is as follows:
Choose the weights for the coordinates of $f$ such that the degrees of $f$ and $g$ are equal. Consider the threefold $X\subset \Ps(w_1,w_2,1,1,1)$  given by the vanishing of $f(x_0,x_1)+g(y_0,y_1,y_2)$.
It turns out that our assumptions on the $\delta_{\alpha}$ imply that the MHS on $H^4(X)$ is of pure $(2,2)$-type. From this purity statement we deduce that $H_4(X)$ can be generated by classes of Weil-divisors on $X$. Then we prove a variant of the Shioda-Tate formula to show that the rank of the Jacobian equals $h^4(X)-1$.
The purity statement and our formula for $h^4(X)$ use several Thom-Sebastiani type results. 

We have several applications of this result.
A \emph{toric decomposition of type $(p,q)$} of $C=V(g)$ consists of two forms $f_1,f_2$ such that $f_1^p+f_2^q=g$.
A \emph{quasi-toric decomposition of type $(p,q,r)$} consists of three co-prime forms $f_1,f_2,f_3$ such that $f_1^p+f_2^q=f_3^rg$.
The quasi-toric decompositions of type $(p,q,\lcm(p,q))$ are precisely the ``finite" $\C(s,t)$-rational points of the curve $x_1^p+x_2^q+g(s,t,1)$. (The polynomial $x_1^p+x_2^q+g(s,t,1)$ defines an affine curve over $\C(s,t)$. It has a unique smooth projective model. The additional points are called the points at infinity. One easily checks that these additional points are all $\C(s,t)$-rational.)

If $H$ is the curve given by $x_1^p+x_2^q+g(s,t,1)$, then we have the following chain of  inclusions:
\[\left\{\begin{array}{c} \mbox{quasi-toric decompositions of } \\ C \mbox{ of type } (p,q,\lcm(p,q)) \end{array} \right\} \subset \Jac(H)(\C(s,t)) \subset H_4(X,\Z)_{\prim}.\]
Our main result gives a sufficient (and necessary, see below) condition, in order that the second inclusion yields a subgroup of finite index. The first inclusion is very often strict, but the rational points of $H$ can be used to study the Jacobian of $H$ and vice-versa. In the case that $(p,q)\in \{(2,3),(2,4),(3,3)\}$ we have that  $H/\C(s,t)$ is an elliptic curve and hence $H\cong \Jac(H)$. Therefore the first inclusion is an equality. Moreover, we can use the theory of elliptic surfaces to define a height pairing on the quasi-toric decompositions of $C$. We call the lattice obtained in this way the \emph{Mordell-Weil lattice}.

Depending on the value of $(p,q)$ we have that the vanishing of $\delta_{1/6},\delta_{1/4}$ or $\delta_{1/3}$ implies that  the height pairing is invariant under equisingular deformations. Hence we can use this pairing to detect Zariski pairs. For example:
\begin{theorem}\label{thmZar}
 There exists an Alexander equivalent Zariski pair $(C_1,C_2)$ of degree 12 curves with 30 cusps and Alexander polynomial $t^2-t+1$.
\end{theorem}

If $C$ is of degree $6k$ then the $(2,3)$-toric decompositions of $C$ correspond to the vectors of length $2k$ in the Mordell-Weil lattice, and there are no non-zero vectors of shorter length. In the case of sextics ($k=1$) we determine the possible Mordell-Weil lattices:
\[
 \begin{array}{|ccc|}
\hline
 \# \mbox{Cusps} & \mbox{Mordell-Weil lattice} & \# \mbox{Shortest vectors/}\\
 &&\#\mbox{toric decompositions}\\
\hline
\leq 5 \mbox{ or } 6 \mbox{ not on a conic} &0&0\\
6 \mbox{ on a conic or }7 & A_2 & 6\\
8 & D_4 & 24\\
9 & E_6 & 72\\
\hline
\end{array}
\]

Moreover, in the sextic case we discuss the relation between  our results and Degtyarev's proof \cite{DegOkaConj,DegOkaConja} of part of Oka's conjecture.
We use Degtyarev's proof to conclude that  if $g\in \C[s,t]$ is a reduced polynomial of degree at most 6 then the group $E(\C(s,t))$ of $\C(s,t)$-points on $y^2=x^3+f$ is generated by points from $E(\C[s,t])$. There are counterexamples to this statement if we do not require $\deg(f)\leq 6$. (Actually, we use the failure of this statement in our construction of the Zariski pair in degree 12.)

Our final application consists of a result on elliptic surfaces, namely:
\begin{theorem} Let  $L=\C(s)$ the function field of $\Ps^1$. Let $f\in L[t]$ be an irreducible polynomial. Let $A,B\in \C$ be such that $4A^3+27B^2\neq 0$. Then the rank of $E_i(L(t))$ is zero for 
 \[ E_1:y^2=x^3+f(t)^2,\; E_2:y^2=x^3+f(t)x \mbox{ and } E_3:y^2=x^3+Af^2x+Bf^3.\]
\end{theorem}
A similar result holds if $L$ is algebraically closed, but for many non algebraically closed fields $L$ the above statement is false.
We prove a similar result on isotrivial families of abelian varieties:
\begin{theorem}
Let $A/\C(s,t)$ be an abelian variety (in particular, with   neutral element $O$ in $A(\C(s,t))$) such that there exists a finite cyclic extension $K/\C(s,t)$ of prime power degree $p^n$ and such that
\begin{enumerate}
\item there exists an abelian variety  $A_0/\C$ with $A_K\cong (A_0)_K$; 
\item the ramification divisor $K/\C(s,t)$ is a prime divisor and
\item the group $A(\C(s,t))$ is finitely generated.
\end{enumerate}
Then $A(\C(s,t))$ is finite.
\end{theorem}

The paper is organized as follows. In Section~\ref{secQua} we recall some facts on the ideals of quasiadjunction and show the purity statement on the MHS of $H^4(X).$ In Section~\ref{secMW}
we establish  the relation between the Alexander polynomial and the Mordell-Weil rank and  prove Theorem~\ref{thmMain}. In Section~\ref{secHei} we recall some facts on the height pairing of elliptic surfaces. In Section~\ref{secTor} we use the height pairing to bound the number of toric decomposition of plane curve. In Section~\ref{secZar} we use the height pairing to establish a Zariski pair of curves with the same Alexander polynomial and we prove Theorem~\ref{thmZar}.
In Section~\ref{secEll} we discuss our application to elliptic surfaces.

\section{Ideals of quasi-adjunction}\label{secQua}
Let $X=V(f)$ be a hypersurface in weighted projective space $\Ps(\bw)$ with coordinates $z_0,\dots,z_{n+1}$.
Let $\Sigma$ be the set of points on $X$ where $X$ is not quasismooth. Assume that $\Sigma$ is finite.
Consider the differential form
\[ \frac{g}{f}\left( \prod z_i \right) \sum (-1)^i w_i \frac{dz_0}{z_0} \wedge \dots \wedge \widehat{\frac{dz_i}{z_i}}\wedge\dots \wedge \frac{dz_{n+1}}{z_{n+1}}.\]
where $g$ is a homogeneous polynomial such that $\deg(g)=\deg(f)-\sum w_i$.
The residue of this form defines an $n$-form on $X\setminus \Sigma$. To decide whether this form can be extended to an arbitrary resolution of singularities of $X$, one has to check the so-called adjoint conditions at every point $p$. This are local conditions, hence the polynomials $g$ satisfying these conditions form an ideal, which we call $I^{(1)}$. 

Let $d$ be the degree of $f$ and let $e$ be a divisor of $d$. For $1\leq k \leq e-1$ we define $I^{(k/e)}\subset K[z_0,\dots,z_{n+1}]$ as follows:
Consider the hypersurface $X_e$ given by $t^e+f=0$ in $\Ps(d/e,\bw)$. Now an element $g$ of $\C[z_0,\dots,z_{n+1}]$ is in $I^{(k/e)}$ if and only if $t^{e-k-1} g$ is in the adjunction ideal of $t^e+f$. 
Note that $I^{(k/e)}$ is a saturated ideal, and that $V(\sqrt{I^{(k/e)}})\subset \Sigma$.

Libgober  showed that in the case $\Ps(\bw)=\Ps^{n+1}$ the ideal $I^{(k/e)}$ is $kd/e-n-1$-regular \cite[Corollary 4.2]{LibAdjHS}. 
Let us denote by $\delta_{k/e}$  the defect of $I^{(k/e)}$ in degree $kd/e-n-2$, i.e., the difference between the Hilbert polynomial and the Hilbert function of $I^{(k/e)}$.  Libgober \cite[Theorem 4.1]{LibAdjHS} showed that \[h^{n,0}(\tilde{X_e})=\sum_{k=1}^{e-1} \delta_{k/e}\]
and that 
\[\Delta^{n,0}(t)= \prod_{k=1}^{d-1} \left(t-\exp\left(\frac{2k\pi \sqrt{-1}}{d}\right)\right)^{\delta_{k/d}}\]
is a divisor of the Alexander polynomial of $X$.
In the curve case  ($n=1$) Libgober proved the equality
\[ \Delta(t)=\Delta^{1,0}(t)\overline{\Delta^{1,0}(t)} (t-1)^{r-1},\]
where $r$ is the number of irreducible components of $X$. (See \cite{LibPlane}.)

Next, we recall some results on the cohomology of the Milnor fiber of an isolated singularity.
The following results can be found in \cite[Section II.8]{Kuli}, unless stated otherwise.
Let $f(x_1,x_2)=0$ be a weighted homogeneous isolated curve singularity. Let $F$ be the Milnor fiber of $f$ and consider $H^1(F)$. There is a natural MHS on this vector space. Let $T$ be the monodromy operator. Then the  $T$-invariant subspace of $H^1(F)$ coincides with the $(1,1)$-part of $H^1(F)$. Since the singularity is weighted homogeneous we get that $T$ is semi-simple and  acts on $W_1H^1(F)$. All eigenvalues of $T$ on $W_1H^1(F)$ are different from 1.

Let $g(y_0,y_1,y_2)$ be a reduced homogeneous polynomial. Then its zero-set $C$ is a curve in $\Ps^2$ with isolated singularities.
There are two equivalent ways to define the Alexander polynomial of $C$. The first definition (e.g., see \cite[Definition 4.1.19]{Dim}) is

\begin{definition}
The \emph{Alexander polynomial} of a reduced plane curve $C=V(g)\subset \Ps^2$ is the characteristic polynomial of the monodromy acting on $H^1(G)$, where $G$ is the Milnor fiber of $g$.
\end{definition}
There is a  second, purely group-theoretic, definition of the Alexander polynomial, which depends on the choice of an epimorphism  $\pi_1(\Ps^2\setminus (C\cup \ell)) \to \Z$, where $\ell$ a general line. (E.g., see \cite{LibPlane}.)

\begin{definition}
Define the join of  $f(x_1,x_2)$ and $g(y_0,y_1,y_2)$ to be the polynomial $f(x_1,x_2)+g(y_0,y_1,y_2)$. If $F$ is the Milnor fiber of $f$ and $G$ is the Milnor fiber of $G$, then we denote by $F\oplus G$ the Milnor fiber of the join of $f$ and $g$.
\end{definition}

In the sequel, we need to know the dimension of the subspace of $H^3(F\oplus G)$ left invariant by $T$ together with its mixed Hodge structure. Let us recall 
\begin{lemma}\label{lemTSposdim}  Suppose $f(x_1,\dots,x_{n+1})$ has an isolated singularity at the origin and  $g(y_1,\dots,y_m)$ has an arbitrary singularity at the origin. Let $F$ and $G$ be  the corresponding Milnor fibers. Then there is an isomorphism
 \[ \tilde{H}^{n+k+1}(F\oplus G,\Q)\cong \tilde{H}^n(F,\Q)\otimes \tilde{H}^k(G,\Q)\]
that respects the monodromy.
\end{lemma}
\begin{proof}
 This follows from \cite[Lemma 3.3.20]{Dim}.
\end{proof}

For an isolated singularity $f(x_1,\dots,x_{n+1})=0$ we have the so-called (Steenbrink) spectrum of $f$. This is a formal sum of rational numbers. The multiplicity $
\nu(\alpha)$ of a rational number $\alpha$ in the spectrum equals the dimension of the $\zeta(-\alpha)$-eigenspace of the semi-simplification of the monodromy operator acting on $\Gr_F^{\lfloor{n-\alpha} \rfloor} H^n(F)$. The spectrum is symmetric around $(n-1)/2$.
Note that $\nu(\alpha)$ is zero outside the interval $(-1,n)$. If $f$ is weighted homogeneous with weights $w_i$ and degree $d$, and $M(f)$ is the Milnor algebra of $f$, then we have
\[ \nu(\alpha)=\dim M(f)_{(\alpha+1) d-\sum w_i}.\]
\begin{example} The spectrum of $f=x_1^2+x_2^3$ equals $(-\frac{1}{6})+(\frac{1}{6})$.
\end{example}

\begin{lemma}\label{lemTS} Let $f(x_1,x_2)=0$ be a weighted homogeneous isolated curve singularity. Let $g(y_0,y_1,y_2)$ be a reduced homogeneous polynomial and $C=V(g)\subset \Ps^2$.  Then $H^3(F\oplus G)^T$ has dimension
\[ \sum_{0\leq  \alpha < 1} (\nu(\alpha-1)+\nu(\alpha))(\delta_{\alpha}+\delta_{1-\alpha})=\sum_{0\leq \alpha<1} (\nu(\alpha-1)+\nu(\alpha)) \ord_{t=\zeta(\alpha)}\Delta_C(t) \]
\end{lemma}

\begin{proof}
For $\alpha\in [0,1)$ let us denote by $H^1(F)_{\alpha}$ the $\zeta(\alpha)$-eigenspace of the monodromy operator. Then we have
\[ H^3(F\oplus G)^T=\bigoplus_{0\leq \alpha <1} H^1(F)_{1-\alpha} \otimes H^1(G)_{\alpha}.\]

Now $\dim H^1(G)_{\alpha}$ equals the order of vanishing of the Alexander polynomial of $C$ at $t=\zeta(\alpha)$, which is $\delta_{\alpha}+\delta_{1-\alpha}$.

From the definition of the spectrum it follows that $\dim H^1(F)_{\alpha}$ equals $\sum_{k\in \Z} \nu(-\alpha+k)$. Since $f$ is an isolated curve singularity we find that the only nonzero contribution may occur at $k=0,1$. Combining everything we get
\[ \sum_{0\leq \alpha<1} (\nu(\alpha)+\nu(\alpha-1))(\delta_{\alpha}+\delta_{1-\alpha}).\] 
\end{proof}

\begin{lemma} Let $f,g$  be two curve singularities. If $(X,p)$ is locally the join of $f$ and $g$,  then the local cohomology group $H^4_p(X)$ has a MHS of pure weight $4$.
\end{lemma}
\begin{proof}
Let $F$ be the Milnor fiber of a  curve singularity. Then the only nontrivial weights of the MHS on $H^1(F)$ are $1$ and $2$ and there is a natural identification between $\Gr_W^2 H^1(F)$ and $H^1(F)^T$.

Scherk and Steenbrink showed a Thom-Sebastiani property for the spectrum of the join of two isolated singularities. (See \cite[Cor 8.12]{SchSte}.) 
Moreover, they expressed the Hodge numbers of the MHS on the cohomology of the Milnor fiber $F\oplus G$ in terms of the Hodge numbers on the cohomology of the two singularities and the action of the monodromy operator.
 
 In our case we find that the only  non-trivial Hodge weights of  $H^3(F \oplus G)$ are $3$ and $4$ and that $\Gr_W^4 H^3(F\oplus G)$ is isomorphic to $H^3(F\oplus G)^T$. From $h^{i,j}(H^4_p(X))\leq h^{i,j}(H^3(F)^T)$ (see, e.g., \cite[Equation (1.6)]{DimDif}) it follows now that $H^4_p(X)$ is of pure weight 4.
\end{proof}

\begin{lemma}  Let $f(x_1,x_2)=0$ be a weighted homogeneous isolated curve singularity. Let $g(y_0,y_1,y_2)$ be a reduced homogeneous polynomial. Assume that the weighted degree of $f$ divides $\deg(g)$. Let $X\subset \Ps(w_1,w_2,1,1,1)$ be $V(f\oplus g)$. Then the MHS on $H^4(X)$ is of pure weight 4.
\end{lemma}
\begin{proof}
Let $X$ be the threefold corresponding to $f(x_1,x_2)+g(y_0,y_1,y_2)=0$. Let $\Sigma$ be the locus where $X$ is not quasismooth. Then $\Sigma$ is contained in $V(x_1,x_2)$. At every point $p\in \Sigma$ we have that the singularity is the join of two curve singularities. In particular $H^4_p(X)$ is of pure weight 4.

Recall that $H^4(X)$ fits in exact sequence
 \[ \to H^4_{\Sigma}(X)\to H^4(X)\to H^4(X\setminus \Sigma) \to \dots\]
From $H^4(X\setminus \Sigma)=\C(-2)$ it follows now that  $H^4(X)$ has pure weight 4.
\end{proof}

For the reader's convenience we remark that in the sequel $f$ is a weighted homogeneous polynomial of weighted degree $d$, where we assume that the weights are integral, but not necessarily co-prime.

\begin{proposition}\label{prpHodgeNumber} Let $f(x_1,x_2)$ be a weighted homogeneous polynomial of weighted degree $d$ with integral weights, such that $f$ has an isolated singularity at the origin. Let $g(z_0,z_1,z_2)$ be a reduced  homogeneous polynomial of degree $d$.

Consider the threefold $X \subset \Ps(w_1,w_2,1,1,1)$  given by $f+g=0$.
Then  the MHS on $H^4(X)$ is pure of weight 4 and we have $ h^{4,0}(X)=0$,
\[ h^{3,1}(X)=\sum_{0\leq \alpha<1} \nu(-\alpha)\delta_\alpha \mbox{ and } h^{2,2}(X)=
1+\sum_{0\leq  \alpha <1} (\nu(\alpha-1)\delta_{\alpha}  +\nu(\alpha)\delta_{1-\alpha})\]
\end{proposition}
\begin{proof}
Let $U$ be the complement of $X$.  Let $F$ and $G$ the affine Milnor fibers of $f$  and $g$ respectively. Using Poincar\'e duality and the Gysin exact sequence for cohomology with compact support we get 
\[ h^3(F\oplus G)^T=h^3(U)=h^5_c(U)=h^4_c(X)-1=h^4(X)-1.\]
Hence we can determine $h^4(X)$ using Lemma~\ref{lemTS}. From the previous lemma it follows that $H^4(X)$ has a pure weight 4 Hodge structure.
Let $\tilde{X}$ be a resolution of singularities of $X$, with exceptional divisor $E$.
Then we have an exact sequence
\[ H^3(E) \to H^4(X) \to H^4(\tilde{X}) \to H^4(E) \to H^5(X), \]
see \cite[Corollary-Definition 5.37]{PSbook}.
Since the Hodge weights of $H^3(E)$ are at most 3 and $H^4(X)$ is of pure weight 4,  we find that $H^4(X)\to H^4(\tilde{X})$ is injective. Since $H^4(E)$ is of pure $(2,2)$ type we get    $h^{p,4-p}(H^4(X))=h^{p,4-p}(\tilde{X})=h^{3-p,p-1}(\tilde{X})$ for $p\neq 2$.
In particular, $h^{4,0}=h^{0,4}=0$ and $h^{3,1}=h^{1,3}=h^{2,0}(\tilde{X})$.

Now \[h^{3,1}(H^4(X)_{\prim})=h^{3,1}(H^5_c(U))=h^{1,3}(H^3(U))=h^{1,3}(H^3(F\oplus G)^T)\]
The first equality follows from the Gysin sequence for cohomology with compact support, the second from Poincar\'e duality ($U$ is a $\Q$-homology manifold).

Hence to prove our result we need to determine $h^{1,3}(H^3(F\oplus G)^T)$. Lemma \ref{lemTSposdim} gives an isomorphism $(H^3(F\oplus G),T_{F\oplus G})\cong(H^1(F)\otimes H^1(G),T_F\otimes T_G).$ However, the mixed Hodge structure is harder to determine. 

We are first going to determine the MHS on $H^3(U')^T$, where $U'=\{f+g=1,f\neq0,g\neq0\}$ is an open subset of the affine Milnor fiber.
Consider the affine curve $C$ given by $t^e+s^d=1,t\neq 0,s\neq 0$. Then there is a natural map $C\times F\times G\to U'$ given by $(t,s,x_1,\dots,x_n,y_1,\dots,y_m)$ mapping to $(tx_1,\dots,tx_n,sy_1,\dots,sy_m)$. This map is Galois, with group $H:=\Z/d\Z\times \Z/e\Z$. 

Let $C'$ be $Z(x^d+y^e+z^e)$ in $\Ps(e/d,1,1)$. Then $C'$ is smooth and is the projective completion of $C$.
There is a natural $H$-action on both $C$ and $C'$.
The cohomology of $C'$ together with the Hodge filtration and the $H$-action can be easily described using the Griffiths-Steenbrink identification with the Jacobian ring of $C'$.
The Gysin exact sequence for compact support now shows that $H^1_c(C)$ is an extension of $H^1_c(C')$ by $\C(0)^{2d+e-1}$ (as MHS). 
From this we obtain that $h^1(C)=de+1$. Moreover, we find that $H^1(C)=\C[H]\oplus \C$ as $\C[H]$-modules.
Fix generators for $H$ such that $(1,0)\in H$ acts by multiplying $s$ with $\zeta(1/d)$ and leaving $t$ invariant and such that $(0,1)\in H$ acts by multiplying $t$ with $\zeta(1/e)$ and leaving $s$ invariant. Let $w=e/d$.

For integers $0\leq a<d$ and $0\leq b <e$  let $H^1(C)_{a,b}$ be the intersection of $\zeta(a/d)$-eigenspace of $(1,0)\in H$ and the $\zeta(b/e)$-eigenspace of $(0,1)\in H$. 

Then the MHS on $H^1(C)_{a,b}$  is of pure $(p,q)$-type with
\[ (p,q)=\left\{ \begin{array}{ll} (1,0) & \mbox{if }wa+b<d\\ (1,1)&  \mbox{if } wa+b=d \mbox{ or }a=b=0\\
(0,1) & \mbox{if }wa+b>d.\end{array}\right.\]
Multiplying $x_i$ by $\zeta(w_i/d)$ and $y_j$ by $\zeta(1/e)$ resp. gives a $H$-action on $F\times G$.

Consider again the map $C\times F\times G \to U'$ given by
\[(s,t,x_0,x_1,y_0,y_1,y_2)\mapsto (sx_0,sx_1,ty_0,ty_1,ty_2).\]
Let $T$ be the monodromy operator, where the orientation is chosen such that $T$ maps $x_i$ to $\zeta(w_i/d)x_i$ and $y_j$ to $\zeta(1/d)y_j$. Then we have
\[H^a(U')=\bigoplus_{i+j+k=a} (H^i(C)\otimes H^j(F)\otimes H^k(G))^H\]
and 
\[  H^a(U')^T=\bigoplus_{i+j+k=a}\bigoplus_{\alpha\in \frac{1}{d}\Z/\Z} (H^i(C)\otimes H^j(F)_{\alpha}\otimes H^k(G)_{1-\alpha})^H.\]
We have
\[(H^i(C)\otimes H^j(F)_{\alpha}\otimes H^k(G)_{1-\alpha})^H=H^i(C)_{(-\alpha,\alpha-1)}\otimes H^j(F)_{\alpha} \otimes H^k(G)_{1-\alpha}\]

Suppose first that $\alpha \neq 0$. Then $H^i(C)_{(-\alpha,\alpha-1)}=0$ for $i\neq 1$ and $H^1(C)_{(-\alpha,\alpha)}=\C(-1)$.
Hence the sum over all $\alpha$ with $\alpha\neq 0$  contributes 
\[ \sum_{\alpha \neq 0} h^{1,0}(H^1(F)_{\alpha})h^{1,0}(H^1(G)_{1-\alpha})=\sum_{0<\alpha<1} \nu(\alpha-1) \delta_{1-\alpha}=\sum_{0<\alpha<1 } \nu(-\alpha) \delta_\alpha \]
 to $h^{3,1}(H^3(U'))$.

In the case $\alpha=0$ then $H^1(F)_{\alpha}$ has type $(1,1)$ and $H^0(F)_{\alpha}$ has type $(0,0)$.
Similarly $H^0(C)_{-\alpha,\alpha-1}$ has type $(0,0)$ and $H^1(C)_{-\alpha,\alpha-1}$ has type $(1,1)$.
For the third factor we have that $H^i(G)_\alpha$ is of type $(i,i)$ for $i=0,1$ and for $i=2$ we have weights $(2,1),(2,2)$ and $(1,2)$. Hence there is no contribution to $h^{3,1}$ for $\alpha= 0$.

Since $\delta_0=0$ it suffices now to prove that \[h^{3,1}(H^(F\oplus G)^T)=h^{3,1}(H^3(U')^T).\]
Consider now the map $H^3(F\oplus G)\to H^3(U')$. We are going to prove that this map is an isomorphism on the $(3,1)$-part of the parts fixed by $T$. I.e., we want to show that the kernel and cokernel of
\[ H^3(F\oplus G)^T\to H^3(U')^T\]
have no $(3,1)$-parts. Applying Poincar\'e duality on the fourfolds $U'$ and $F\oplus G$, this is equivalent by showing that the kernel and cokernel of
\[  H^5_c(U')^T\to H^5_c(F\oplus G)^T\]
have no $(1,3)$-parts. 

Let $\Sigma$ be $V(f,g-1)\cup V(f,g-1)$. Then $\Sigma=(F\oplus G)\setminus U'$. 
The above mentioned kernel and cokern are quotient of resp. subspace of
\[ H^4_c(\Sigma)^T \mbox{ and } H^5_c(\Sigma)^T.\]

Now $\Sigma$ has two connected components, namely $f=0,g=1$ and $f=1,g=0$. 
Note that $f=0$ is a cone over $d$ points. If we leave out the vertex then we get a space on which the action of $T$ on the cohomology with compact support is trivial. Hence also on $H^i_c(\{f=0\})$ we do have a trivial $T$ action. By a similar reasoning we find that $T$ acts trivially on the cohomology of $g=0$. By the K\"unneth formula we have
\[ H^k_c(\Sigma)^T=\bigoplus_{i+j=k} H^i_c(\{f=0\})\otimes H^j(G)^T \bigoplus \bigoplus_{i+j=k} H^i_c(F)^T\otimes H^j_c(\{g=0\})\]
The only Hodge types of the cohomology of $f=0$ are of the form $(p,p)$ and since the Hodge type of the cohomology of $G$ are all of the form $(q,q)$, $(q,q\pm 1)$ we do not have a contribution of the first summand to $h^{3,1}$. 

Similarly for $H^i_c(F)^T$ we have only possible Hodge types are of the form $(p,p)$, and since the cohomology of $g=0$ consists of cohomology pulled back form the projective curve defined by it, and some classes of type $(q,q)$, we get only classes of type $(p,p\pm 1)$ and $(p,p)$.
Hence the $h^{3,1}$ and $h^{1,3}$ of the above mentioned kernel and cokernal are zero and we are done.
\end{proof}

\begin{remark}
The above strategy can be used to describe the MHS on $H^k(F\oplus G)$ in terms of $H^i(F)$ and $H^j(G)$. However, a complete description is more cumbersome than the description of the part we needed for our proof.

We can also apply this strategy for the join of two isolated quasihomogeneous singularities and we recover the result of Scherk and Steenbrink \cite{SchSte}.
\end{remark}

\begin{remark}\label{rmkSimplify}
Using that $\nu(\alpha)=\nu(-\alpha)$ and $\nu(-1)=0$ we get
\[ \sum_{0\leq  \alpha <1} (\nu(\alpha-1)\delta_{\alpha}  +\nu(\alpha)\delta_{1-\alpha})=\nu(0)\delta_1+2\sum_{0<\alpha<1} \nu(\alpha)\delta_{1-\alpha}\]
\end{remark}

\begin{remark}\label{rmkHS}
 Using  $\nu(\alpha)=\nu(-\alpha)$ again we get that if $\sum \nu(\alpha)\delta_\alpha=0$ then the MHS on $H^4(X)$ is of pure $(2,2)$-type.
\end{remark}

\begin{example}
Suppose $f=x_0^2+x_1^3$. Then $\nu(-1/6)=\nu(1/6)=1$ and $\nu(\alpha)=0$ for any other $\alpha$.
Hence $h^{3,1}=0$ if and only if $\delta_{1/6}=0$. Moreover, $h^{2,2}=2\delta_{5/6}$.

Similarly if $f=x_0^2+x_1^e$ then $h^{3,1}=0$ if and only if $\delta_{1/2-k/e}=0$ for $k=1,\dots,e-1$.

In these case we have
\[ h^{2,2}=1+2 \sum_{k=1}^{(e-1)/2} \delta_{1/2+k/e}\]
if $e$ is odd and
\[ h^{2,2}=1+\delta_1+2 \sum_{k=1}^{(e-2)/2} \delta_{1/2+k/e}\]
if $e$ is even.
\end{example}

\section{Mordell--Weil ranks of Jacobians of isotrivial families of curves}\label{secMW}

Let $C=\{g(z_0,z_1,z_2)=0\}$ be a reduced curve of degree $d$. 
Let $f(x_1,x_2)$ be a weighted homogeneous polynomial with integral weights $w_1, w_2$ such that the weighted degree of $g$ is $d$.
Then
\[f(x,y)+g(s,t,1)=0\]
defines a curve $H$ over $\C(s,t)$. Let $J$ be its Jacobian. We will now discuss how one can determine the rank of $J(\C(s,t))$ if certain $\delta_{\alpha}$ vanish.
Consider now the threefold $X \subset \Ps(w_1,w_2,1,1,1)$ defined by
\[ f(x_1,x_2)+g(z_0,z_1,z_2).\]
We will prove the following result:
\begin{theorem}\label{thmMW} 
Let $f,g,H$ and $X$ be as above. 
Assume that $\sum_{0\leq\alpha<1} \nu (\alpha)\delta_{\alpha}$ vanishes.
Then the rank of $J(\C(s,t))$ equals
\[ 
2 \sum_{0< \alpha <1}  \nu(\alpha)\delta_{1-\alpha}=\sum_{0<\alpha<1} (\nu(\alpha)+\nu(\alpha-1) \ord_{t=\zeta(\alpha)}\Delta_C(t).\]
\end{theorem}
\begin{remark} If $C$ has only semi-quasihomogeneous singularities then the ideals  $I^{\alpha}$, and therefore their Hilbert function,  can be determined effectively. For arbitrary curve singularities these ideals are well-understood.
\end{remark}

\begin{remark} If $g=x_1^2+x_2^3$ then $J(H)\cong H$. Moreover if $C$ is a cuspidal curve then $\delta_{1/6}=0$ and hence $\sum \nu(\alpha)\delta_\alpha=0$. Hence we recover the main result of \cite{CogLib}.
\end{remark}

Consider the rational map $\Ps(w_1,w_2,1,1,1) \dashrightarrow \Ps^2$, defined by sending $(x,y,z_0,z_1,z_2)$ to $(z_0,z_1,z_2)$. This map can be made into a morphism by blowing up the line $\ell_\infty$ given by $z_0=z_1=z_2=0$. Let $X'$ be the strict transform of $X$ and $\psi:X'\to \Ps^2$ be the induced morphism. 
Let $\tilde{X}$ be a resolution of singularities of $X'$.
Let $\pi:\tilde{X}\to \Ps^2$ be the composition of $\tilde{X}\to X'$ with $\psi$.

\begin{proposition}\label{prpPic} 
Assume that $\sum_{0\leq\alpha<1} \nu (\alpha)\delta_{\alpha}=0$.

Let $\epsilon$ be the number of points in the intersection of $\ell_\infty$ and $X$. (Equivalently, $\epsilon$ is the number of branches of $f=0$ at the origin). Let $r$ be the number of irreducible components of the exceptional divisor of $\tilde{X}\to X'$. Then
 \[ \rank \Pic(\tilde{X})=h^4(X)+\epsilon+r\]
\end{proposition}

\begin{proof}
If $\pi:Y\to X$ is a proper modification with center $\Delta$ and exceptional divisor $E$ we have an exact sequence \cite[Corollary 5.37]{PSbook}
\begin{equation}\label{eqnMV} H^i(X)\to H^i(Y)\oplus H^i(\Delta)\to H^i(E)\to H^{i+1}(X). \end{equation}
We apply this sequence to the modifications $X'\to X$ and $\tilde{X}\to X'$ and their composition. 

First note that the exceptional divisor $E$ of $\tilde{X}\to X$ is a union of surfaces. Hence $H^4(E)$ is pure of type $(2,2)$. 
From Remark~\ref{rmkHS} it follows that the mixed Hodge structure on $H^4(X)$ is of pure $(2,2)$-type.
Hence the mixed Hodge structure on $H^4(\tilde{X)}$ is pure of type $(2,2)$.
 By \cite[Theorem 5.2.11]{Dim} we obtain that $H^5(X)$ vanishes. The exact sequence (\ref{eqnMV}) yields  $H^5(\tilde{X})=0$, hence $H^1(\tilde{X})=0$ holds by Poincar\'e duality. In particular,
\[ \rank \Pic(\tilde{X})=h^2(\tilde{X})=h^4(\tilde{X}).\]
The line $\ell_\infty$ intersects $X$ in $\epsilon$ points. Over each of these points the exceptional divisor of $X'\to X$ is irreducible. In other words $\Div(X')=\Div(X)\oplus \Z^\epsilon$ and $\Div(\tilde{X})=\Div(X')\oplus \Z^r$. From (\ref{eqnMV}) one obtains  \[h^4(X')-h^4(X)=\epsilon\mbox{ and } h^4(\tilde{X})-h^4(X')=r.\]
\end{proof}

\begin{definition}
Let $D$ be a prime divisor in $\tilde{X}$. We call $D$ \emph{vertical} if $\pi(D)$ is a point or a curve and we call $D$ \emph{horizontal} if $\pi(D)=\Ps^2$.

Denote by $\Div_v(\tilde{X})$ the subgroup generated by the vertical divisors in $\Div(\tilde{X})$. Denote by $\Pic_v(\tilde{X})$ the subgroup generated by the vertical divisors in  $\Pic(\tilde{X})$.
\end{definition}

\begin{lemma}\label{lemPicVeven} Suppose $C$ has $c$ irreducible components. Then we have
 \[ \rank \Pic_v(\tilde{X})=1+r+c(\epsilon-1).\]
\end{lemma}
\begin{proof}
Let $D\subset \tilde{X}$ be a vertical prime divisor and consider its image $D'$ in $X'$. Since the image of an irreducible component of  exceptional divisors of $\tilde{X}\to X'$ in $X'$ is a point we get that if $D'$ is not a divisor in $X'$ then $D'$ is a point, and $D$ is a component of the exceptional divisor.

If $D'$ is a divisor in $X'$ then its image in $\Ps^2$ is  an irreducible curve $C_0$. If $C_0$ is not a component of $C$ then the fiber of $\psi$ over a general point of $C_0$ is irreducible and $D'=\psi^{-1}(C_0)$. In particular, $D$ is a linear combination of a class from $\pi^*\Pic(\Ps^2)$ and exceptional divisors.

Write $C=\cup C_i$, with $C_i$ irreducible. Then the number of irreducible components of $\psi^{-1}(C_i)$ equals $\epsilon$. The sum of these $\epsilon$ components is an element of $\psi^*\Pic(\Ps^2)$. 
Denote with $D_{i,j}$ the components of $\psi^{-1}(C_i)$. 

From the above discussion we obtain that $\Pic_v(\tilde{X})$ can be generated by the pull back of a line in $\Ps^2$, the $r$ exceptional divisors $E_1,\dots,E_r$ and the $D_{i,j}$ with $j\in \{1,\dots,\epsilon-1\}$. In particular,
\[ \rank \Pic_v(\tilde{X})\leq 1+r+c(\epsilon-1).\]
To show that we have equality we need to show that the above constructed classes are linearly independent.
Suppose there are $a,b_i,c_{i,j}\in \Q$ such that
\begin{equation}\label{relPic} a \pi^*(\ell)+\sum_{i=1}^r b_i E_i+\sum_{i=1}^c\sum_{j=1}^{\epsilon-1} c_{i,j} D_{i,j}=0\end{equation}
in $\Pic(X)$.
In the sequel, we are going to prove that then $a=0$ and $c_{i,j}=0$ for all $i,j$. Since the exceptional divisors and $\pi^*\ell$ are linearly indepedent in $\Pic(\tilde{X})$ this suffices to show that 
\[ \rank \Pic_v(\tilde{X})\geq 1+r+c(\epsilon-1).\]

Pick now a general line $\ell$ in $\Ps^2$. Then $S:=\pi^{-1}(\ell)$ is a smooth surface.
Define now $E_{i,j}:=D_{i,j}\cap S$.
The inclusion $S\hookrightarrow \tilde{X}$ induces a map $\Pic(\tilde{X})\to \Pic(S)$. The exceptional divisors are in the kernel of the above map. Hence the pullback of (\ref{relPic}) is
\[ aF+\sum c_{i,j} E_{i,j}=0\]
where $F$ is a general fiber of the map $S\to \ell$.

We are now going to compute the Gram matrix of the $F,E_{i,j}$ with respect to the intersection pairing: For $i\neq k$ we have that $E_{i,j}$ and $E_{k,m}$ are disjoint. Moreover, for each $i$ we have that the  general fiber $\psi^{-1}(p)$ in a linear combination of the $E_{i,j}$s. 
Note that each of the points ``at infinity'' yields a section $\ell\to S$, choose one of them to be the zero-section. We consider next the subgroup $\Lambda$ of the Picard group of $S$ generated by the class of a fiber, the zero-section, and all fiber components $E_{i,j}$ which do not intersect the zero-section. These are $2+c(\epsilon-1)$ classes and it suffices to show that they are linearly independent, since $1+c(\epsilon-1)$ are images of vertical classes.

Consider the Gram matrix $G$ of the intersection pairing on $\Lambda$. This matrix consists of $c+1$ blocks, i.e., $\Lambda$ is the orthogonal direct sum of $c$ sublattices $\Lambda_i$  corresponding to the singular fibers and one rank two lattice. This latter lattice has an element of positive self-intersection. Hence by the Hodge-index  theorem we get that the other $c$ sublattices do not have elements with positive self-intersection.

We need to show that each $\Lambda_i$ has rank $\epsilon-1$.

Fix an $i$ and consider now the lattice $\Lambda_i'$ spanned by $E_{i,j}$ for all $j$. Then $\Lambda_i'$ contains the lattice $\Z F\oplus^{\perp} \Lambda_i'$ as a lattice of finite index. 
Assume that $E_{i,\epsilon}$ is the unique component intersection the zero section.
The determinant of the Gram matrix on $F,E_{i,1},\dots,E_{i,\epsilon-1}$ is zero. Since $F$ is a linear combination of $E_{i,1},\dots,E_{i,\epsilon}$ we have that the determinant of the Gram matrix on these generators is also zero. 

If $f$ is homogeneous a polynomial of degree $d$ then $E_{i,j}.E_{i,k}=1$ for $j\neq k$. Moreover, $E_{i,j}^2$ is independent of $i$ and $j$. Set $e=E_{i,j}^2$. 
The eigenvalues of the Gram-matrix on $E_{i,1},\dots,E_{i,\epsilon}$ are $e+\epsilon-1$ (with multiplicity one) and $e-1$ (with multiplicity $\epsilon-1$). From the above discussion we know that one of them is zero, i.e. $e=1$ or $e=1-\epsilon$

The  eigenvalues of the Gram-matrix on $E_{i,1},\dots,E_{i,\epsilon-1}$ are $e+\epsilon-2$ (with multiplicity one) and $e-1$ (with multiplicity $\epsilon-2$).
We know that both are nonnegative. In particular $e\leq 2-\epsilon$. Therefore $e=1-\epsilon$. 
In particular the Gram-matrix on $E_{i,1},\dots,E_{i,\epsilon-1}$ has nonzero determinant and therefore $\rank \Lambda_i=\epsilon-1$.
Hence we find that the rank of  $\Lambda$ is $c(d-1)+2$, which finishes the proof in this case.

If $f$ is not homogeneous, then there exist homogeneous polynomials $h_1,h_2$ such that $f(h_1,h_2)$ has an isolated singularity and is homogeneous of degree $d$.  Consider now the threefold $Y$ obtained by blowing up  $f(h_1,h_2)+g=0$ along $\ell_\infty$. Let $S'=Y\cap \pi^{-1}(\ell)$. Then we have a finite morphism $S'\to S$. Hence we find that for each reducible fiber there is at most one relation between the general fiber and the fiber components. Therefore we have 
 $\rank \Pic_v(S)\geq 1+c(\epsilon-1)$ and we are done.
 \end{proof}

 \begin{lemma}\label{lemRnkPic} We have 
  \[ \rank J(H)(\C(s,t)) = \rank \Pic(\tilde{X})-\rank \Pic_v(\tilde{X})-1.\]
 \end{lemma}
\begin{proof}
Let $F$ be a general fiber of $\pi$. Let $\eta$ be the generic point of $\Ps^2$. Let $Z_0$ be the zero-section. Then the map $\Pic(\tilde{X}) \to J(H)(\C(s,t))$ defined by $D\mapsto D_\eta-(D.F) Z_0$ is surjective. 
Let $D$ be a divisor in the kernel, then for some $k$ we have  $D_{\eta}-k(Z_0)_\eta=0$ in $J(H)$. The degree of $D_{\eta}-k(Z_0)_{\eta}$ as a divisor on $H$ is zero, i.e. $\dim |D_{\eta}-k(Z_{0})_{\eta}|=0$. In particular,  there is an open subset $U\subset \Ps^2$ on which we have $D|_{\pi^{-1}(U)}=kZ_0|_{\pi^{-1}(U)}$. This means that $D-kZ_0$ is a sum of prime divisors $D_i$ such that $\pi(D_i)\subset \Ps^2\setminus U$. In particular, each $D_i$ is a vertical divisor. Hence the rank of the kernel of the map $\Pic(\tilde{X})\to J(H)(\C(s,t))$ equals $\rank \Pic_v(\tilde{X})+1$. 
\end{proof}

 \begin{proposition} 
 We have 
 \[ \rank J(H)(\C(s,t)) \leq h^4(X)-(\epsilon-1)(c-1)-1, \]
 where equality if attained if and only if  $\sum \nu(\alpha)\delta_{\alpha}=0$.
   \end{proposition}

\begin{proof}
We have
\begin{eqnarray*}
 \rank J(H)(\C(s,t)) &= &\rank \Pic(\tilde{X})-\rank \Pic_v(\tilde{X})-1\\
&=&\rank \Pic(\tilde{X})-2-r-c(\epsilon-1) \\
&\leq & h^4(\tilde{X})-2-r-c(\epsilon-1)\\
&=& h^4(X')-2-c(\epsilon-1) \\
&=& h^4(X)-2-c(\epsilon-1)+\epsilon \\
&=& h^4(X)-1-(c-1)(\epsilon-1).
\end{eqnarray*}
If $\sum \nu(\alpha)\delta_{\alpha}=0$ then the inequality is an equality. If  $\sum \nu(\alpha)\delta_{\alpha}\neq 0$ then from Proposition~\ref{prpHodgeNumber} we obtain that $h^{3,1}(H^4(\tilde{X}))>0$ and the inequality is strict.
\end{proof}
 
\begin{proof}[{Proof of Theorem~\ref{thmMW}}] 
We have the following equalities:
\begin{eqnarray*} \rank J(H)(\C(s,t))&=&h^4(X)-1-(c-1)(\epsilon-1) \\
 &=& h^{2,2}(X)-1-\nu(0)\delta_1\\
 &=& 2 \sum_{0<\alpha<1} \nu(\alpha)\delta_{1-\alpha}.
\end{eqnarray*}
The first one follows from the previous proposition, the second equality follows from Proposition~\ref{prpHodgeNumber} and the equalities $\nu(0)=\dim H^0(V(f))-1 =\epsilon-1$ and $\delta_1=\ord_{t=1}\Delta_C=c-1$.
The third equality follows from the same proposition combined  with Remark~\ref{rmkSimplify}.

Similarly, we have 
\begin{eqnarray*} \rank J(H)(\C(s,t))&=&h^4(X)-1-(c-1)(\epsilon-1) \\
 &=&\sum_{0\leq \alpha <1} (\nu(\alpha)+\nu(\alpha-1)) \ord_{t=\zeta(\alpha)}\Delta_c(t) -\nu_0\delta_1.
\end{eqnarray*}
\end{proof}

\section{Height pairing for elliptic surfaces}\label{secHei}
In this section let $K$ be a field.
We recall some well known facts on elliptic surfaces over $K$. For the details we refer to \cite[Lecture III]{MiES} and \cite[Lecture VII]{MiES}.
\begin{definition}
 An \emph{elliptic surface} is a smooth projective surface $X$ together with a morphism $\pi:X\to C$ to a smooth projective curve $C$, such that the general fiber is a genus one curve, no fiber contains a $(-1)$-curve, and there is a section $\sigma_0:C\to X$.
 
 A Weierstrass model for $X$ consists of a choice of a line bundle $\cL$ on $C$ and  a hypersurface $W$ in $\Ps(\cO \oplus \cL^{-2}\oplus \cL^{-3})$ such that $W$ is birational to $X$ and $W$ is given by
\[ \{Y^2Z=X^3+AXZ^2+BZ^3\} \subset \Ps(\cO\oplus \cL^{-2}\oplus \cL^{-3})\]
for appropriate $A\in H^0(\cL^4)$ and $B\in H^0(\cL^6)$.
The zero locus of  $4A^3+27B^2$ is called the \emph{discriminant} of this Weierstrass model. A Weierstrass equation is \emph{minimal} if at every place $p$ of $C$ either $v_p(A)<4$ or $v_p(B)<6$ holds.

A $K(C)$-rational point on the generic fiber $E/K(C)$ of $\pi$ yields a rational section $C\dashrightarrow X$, which can be extended to a section $C\to X$, since $C$ is a smooth projective curve. There is a fiberwise addition map, hence the group of sections is an abelian group.
\end{definition}
One easily shows that every elliptic surface with a section has a minimal Weierstrass model.

\begin{lemma}
The Weierstrass model $W$ coincides with $X$ if and only if for every $p \in C$ one of the following two conditions holds 
\begin{enumerate}
 \item $v_p(\Delta)\leq 1$ or
 \item $v_p(\Delta)=2$ and $v_p(A)=1$. 
 \end{enumerate}
\end{lemma}
 The common terminology for this case is an \emph{elliptic surfaces without reducible fibers}.

\begin{definition}
Let $\pi:X\to C$ be an elliptic surface. The \emph{trivial sublattice} $T$ of $\NS(X)$ is the  lattice is generated by the class of a fiber, the class  of the zero-section, and the irreducible components of the reducible fibers not intersecting the zero-section. 
\end{definition}
One can easily check that the above mentioned generators of $T$ are linearly independent in $\NS(X)\otimes \Q$.

\begin{proposition}[Shioda-Tate formula]
There is a natural isomorphism between $\NS(X)/T$ and the Mordell-Weil group $E(K(C))$.
\end{proposition}
After tensoring with $\Q$ we can find a section $\varphi$ to the quotient map $\NS(X)\otimes \Q \to E(K(C))\otimes \Q$, such that the image lands in the orthogonal complement of $T$. The restriction of the intersection pairing to $\varphi(E(K(C))\otimes \Q)$ introduces a negative definite pairing on $E(K(C))/E(K(C))_{\tor}$. Shioda \cite{ShiMW} defined  the height pairing on $E(K(C))$ to be minus this pairing.

In the sequel all elliptic surfaces under consideration  are without reducible fibers. In this case the height pairing is easier to calculate than in the case with reducible fibers:
\begin{proposition} Let $W$ be the Weierstrass model of an elliptic surface without reducible fibers. 
Then  $E(K(C))$ is torsion free and the height pairing is an even positive-definite pairing.

Moreover, let $\chi=\chi(\cO_W)$ be the Euler characteristic of the structure sheaf of $W$.
Let $(X:Y:Z)=(0:1:0)$ be the image of the zero section $Z$. Let $\sigma:\Ps^1\to X=W$ be another section and let $S$ be its image in $X$. Denote  the height pairing by $\langle ,\rangle$ and the intersection pairing on $\NS(W)$ by $(.)$. Then 
\[ \langle S,S\rangle =2\chi+2(S.Z).\]
If $S'$ is a further section, then
\[ \langle S,S'\rangle =\chi+(S.Z)+(S'.Z)-(S.S').\]
\end{proposition}

We now assume that $C=\Ps^1$. Then $\cL=\cO_{\Ps^1}(k)$ for some $k\in \Z_{>0}$. In this case $\chi(\cO_W)=k$ and $W$ is the blow-up (at $(1:1:0:0)$) of a quasi-smooth hypersurface of degree $6k$  in $\Ps(2k,3k,1,1)$.

\section{Heights, toric decompositions and Oka's conjecture}\label{secTor}
In this section we concentrate on the case where $f=x_1^2+x_2^3$. To apply our main result we have to assume that $\deg(g)=6k$.
For historical reasons  we denote the curve $H/\C(s,t)$ by $E/\C(s,t)$.

From our main result it follows that  if $\delta_{1/6}=0$ holds then we can compute the Mordell-Weil rank by determining the Alexander polynomial of $C=V(g)$. 
In this section we use a variant of Shioda's height pairing on elliptic surface to introduce a pairing on $E(\C(s,t))$. We will show that if $\delta_{1/6}=0$ then the height pairing is invariant under equisingular deformations of the plane curve $C$, and hence can be used to detect Zariski pairs. We will do the latter in the next section.

Consider now the elliptic threefold $x_0^2+x_1^3+g(y_0,y_1,y_2)$. Take now a general line $\ell$ of the form $y_2=a_0y_0+a_1y_1$. Then $x_0^2+x_1^3+g(y_0,y_1,a_0y_0+a_1y_1)$ is an elliptic surface with Mordell-Weil group $E_\ell(\C(t))$. The specialization map $E(\C(s,t))\to E_{\ell}(\C(t))$ is injective. We can define the height pairing on $E(\C(s,t))$ as the restriction of the height pairing on $E_\ell(\C(t))$.

\begin{proposition}\label{prpInvariance} Let $C=V(g)\subset \Ps^2$ be a curve such that $6\mid \deg(g)$ and $\delta_{1/6}=0$.  Then the height pairing on $E(\C(s,t))$ is invariant under equisingular deformations of $g$.
\end{proposition}
\begin{proof}
Consider the an equisingular deformation $g_t$ of $g$. Then $\delta_{1/6}=0$ holds along this family.

Fix a general line $\ell$ in $\Ps^2$, and let $X_t$ be the resolution of singularities of $x_0^2+x_1^3+g_t$, and let $S_t=X_t\cap \pi^{-1}(\ell)$. Then $S_t$ is an elliptic surface. The height pairing on $E(\C(s,t))$ is the intersection pairing on the orthogonal complement of the class of the zero-section and of the fiber in the image of $\NS(X_t)$ in $H^2(S_t,\Z)$. Since $\delta_{1/6}$ vanishes we have by Remark~\ref{rmkHS} that $H^2(X_t,\C)$ is of pure $(1,1)$-type and hence that $\NS(X_t)=H^2(X_t,\Z)$. As $H^2(X_t,\Z)$ is invariant under equisingular deformation, the same holds true for its image in $H^2(S_t,\Z)$. Hence  the height pairing is invariant under equisingular deformations.
\end{proof}

A point of $E(\C(s,t))$ can be represented by three coprime forms $f_1,f_2,f_3$ satisfying
\[ \left(\frac{f_1}{f_3^3}\right)^2=\left( \frac{f_2}{f_3^2}\right)^3+g.\]
Hence every point of $E(\C(s,t))$ yields a quasi-toric decomposition of $g$ of type $(2,3,6)$.
From the fact that $g$ is irreducible it follows directly that $f_1,f_2,f_3$ are \emph{pairwise} coprime.

The height can now be calculated as follows:
Let $n=\deg(f_3)$ then $\deg(f_1)=2(k+n)$ and $\deg(f_2)=3(k+n)$. Then by the results of the previous section we get that  height of the point equals $2(k+n)$. A point corresponds to a toric decomposition of type $(2,3)$ if and only if $n=0$, Hence the height is $2k$ in this case and we obtain:
\begin{lemma} Every non-zero element of $E(\C(s,t))$ has height at least $2k$. The elements with height $2k$ correspond one-to-one to the toric decompositions of $g$ of type $(2,3)$.
\end{lemma}

\begin{remark}
Let $\omega$ be a primitive third root of unity. If $(f_1,f_2,f_3)$ is a quasi-toric decomposition then so is $(\omega f_1,-f_2,f_3)$. Hence we get 6 toric decompositions. Many authors consider these decompositions equivalent.
\end{remark}

We can calculate the height pairing of two points similarly.
If we have two points given by the quasi-toric decompositions $(f_1,f_2,f_3)$ and $(g_1,g_2,g_3)$ then the height pairing equals
\[ k+n_1+n_2-\deg \gcd(g_3^3f_1-f_3^3g_1 , g_3^2f_2-f_3^2g_2),\]
for $n_1=\deg(f_3)$ and $n_2=\deg(g_3)$.

The elliptic curve $E$ has an automorphism of order $6$, sending $(x,y)\to (\omega^2 x,-y)$, with $\omega$ a primitive third root of unity. We denote the square of this automorphism by $\omega$. This makes $E(\C(s,t))$ into a $\Z[\omega]$-module.
From the fact that all singular fibers of $S\to \ell$ are irreducible, it follows that $E(\C(s,t))$ is torsion-free.
In particular we have a free $\Z[\omega]$-module and the Mordell-Weil group is of even $\Z$-rank.

Since $\omega$ is an  automorphism we have $\langle \omega P,\omega P\rangle=\langle P,P\rangle$.
\begin{lemma} We have $\langle P,\omega P\rangle =-\frac{1}{2} \langle P,P\rangle$.\end{lemma}
\begin{proof} This follows from
\[ \langle P,\omega P \rangle=\langle \omega P,\omega^2 P\rangle =\langle \omega P,(-\omega-1) P\rangle=-\langle P,P\rangle-\langle P,\omega P\rangle.\]
\end{proof}

\begin{proposition} Suppose $\ord_{t=\zeta(1/6)} \Delta_C=1$. Then $C$ has either no $(2,3)$-toric decompositions or precisely six $(2,3)$-toric decomposition.
\end{proposition}

\begin{proof}
In this case we have that the Gram matrix of the height pairing is \[ \left(\begin{matrix}2(k+n) & -k-n \\ -k-n&2(k+n) \end{matrix}\right)\]
In particular, the lattice is $A_2(k+n)$. The shortest vector in this lattice has length $2(k+n)$ and there are precisely 6 shortest vectors. Hence if there is a toric decomposition then $n=0$ and this decomposition is unique up to the $\mu_6$-action.
\end{proof}

Let us now consider irreducible sextic curves. Then Oka conjectured the following (see \cite{OkaConj}):
\begin{conjecture} Let $C$ be an irreducible sextic plane curve, having no toric decomposition of type $(2,3)$. Then the following properties hold:
 \begin{enumerate}
 \item $\Delta_C(t)=1$;
 \item If $C$ has simple singularities then $\pi_1(\Ps^2\setminus C)\cong \Z/6\Z$ and $\pi_1(\Ps^2\setminus (C\cup \ell))\cong \Z$ for a general line $\ell$;
 \item $\pi_1(\Ps^2\setminus C)\cong \Z/6\Z$ and $\pi_1(\Ps^2\setminus (C\cup \ell))\cong \Z$ for a general line $\ell$.
 \end{enumerate}
 \end{conjecture}
 Degtyarev \cite{DegOkaConj} showed that the second and third statement are false. He proved part one in \cite{DegOkaConj,DegOkaConja}.
Note that if $C$ has a quasi-toric decomposition of type $(2,3,6)$ then the Alexander polynomial is nonconstant. Hence from the first part of Oka's conjecture it follows that if a sextic curves has  a quasi-toric decomposition  then it has  a toric decomposition. In the next section we will see that the corresponding statement is false if $\deg(C)=12$.

Degtyarev proved a stronger statement, relating the number of toric decompositions with the vanishing order of the Alexander polynomial. In order to formulate this, we  note first that
 \begin{lemma} Let $C$ be an irreducible sextic curve with isolated singularities. Let $\alpha\in [0,1)$. If $\alpha\neq 5/6$ then $\delta_{\alpha}=0$.
 \end{lemma}

 \begin{proof}
 From the work of Zariski \cite{ZarIrr} it follows  that for an irreducible plane sextic we have $\delta_{1/3}=\delta_{2/3}=\delta_{1/2}=0$. Hence  the only possible non-zero $\delta_{\alpha}$ are $\delta_{1/6}$ and $\delta_{5/6}$.

If $\delta_{1/6}>0$ then one of the singularities of $C$ has $-5/6$ in its spectrum. However, the spectrum of the surface singularity $x^6+y^6+z^6$ starts at $-1/2$, and hence by the semi-continuity of the spectrum \cite{VarSC} we have  $\delta_{1/6}=0$. 
\end{proof}

In particular, the Alexander polynomial of an irreducible sextic is non-trivial if and only if $\delta_{5/6}>0$.
Degtyarev showed in a previous paper that $\delta_{5/6}\leq 3$ holds. In his proof of the first part of Oka's Conjecture he showed
\begin{theorem}[Degtyarev]\label{thmDeg}
 Let $C$ be an irreducible sextic curve. Then the number of $(2,3)$-toric decomposition of $C$ equals $6$ if $\delta_{5/6}=1$, equals $24$ if $\delta_{5/6}=2$ and equals $72$ if $\delta_{5/6}=3$.
\end{theorem}
Our methods are not sufficient to reprove Theorem~\ref{thmDeg}, but we obtain several related results which are of some independent interest. The first one is a corollary of Theorem~\ref{thmDeg}
\begin{corollary} Let $g$ be an irreducible sextic curve. Let $E/\C(s,t)$ be the elliptic curve $y^2=x^3+g(s,t,1)$. Then $E(\C(s,t))$ is generated by $E(\C[s,t])$.
\end{corollary}

\begin{proof} 
The points in $E(\C[s,t])$ are precisely the neuteral element and the points corresponding to toric decompositions. They generate a sublattice $\Lambda$ of the Mordell--Weil lattice, which is a root lattice of even rank. Moreover, $\Lambda$ contains all the shortest vectors of the Mordell--Weil lattice. Hence, if $\Lambda$ has finite index in the Mordell--Weil lattice then it is the Mordell--Weil lattice.

Since $\delta_{1/6}=0$ our main result implies  that the rank of $E(\C(s,t))$ equals $2\delta_{5/6}$. 

If $\delta_{5/6}=1$ then we know from Degtyarev's results that the rank of $\Lambda$ is positive and even. Since the Mordell-Weil lattice has rank 2 we are done in this case.

If $\delta_{5/6}=2$ then we know that $\Lambda$ is a root lattice with $24$ shortest vectors, of even rank, and rank at most 4. From the classification of root lattices we obtain that then $\rank \Lambda=4$ and we are done.

Similarly, if $\delta_{5/6}=2$ then we know that $\Lambda$ is a root lattice with $72$ shortest vectors, of even rank, and rank at most 6. From the classification of root lattices we obtain that then $\rank \Lambda=6$ and we are done.
\end{proof}

We are now going to describe the Mordell-Weil lattices of elliptic threefolds associated with cuspidal sextics:
\begin{proposition} Let $C$ be a sextic curve with $8$ cusps. Then $C$ has either $6$ or $24$ toric decompositions. In the latter case the Mordell-Weil lattice is $D_4$.
\end{proposition}
\begin{proof}
 Suppose we have more than $6$ toric decompositions, say $g=g_1^3+f_1^2=g_2^3+f_2^2$, where the $f_i$ are cubics, the $g_i$ are quadrics and $f_1 \neq \pm f_2$. Then we find
\[ (f_1-f_2)(f_1+f_2)=(g_1-g_2)(g_1-\omega g_2)(g_1-\omega^2 g_2)\neq 0\]
We can then find six forms $u_0,u_1,u_2,v_0,v_1,v_2$ such that $f_1-f_2=u_0u_1u_2$, $f_1+f_2=v_0v_1v_2$ and $g_1-\omega^j g_2=u_jv_j$.

From the fact that the degrees of the $u_i$s and $v_j$s are at most $2$ it follows that it suffices to consider two cases, namely when $\deg(u_2)=0$ and the case where all $u_i,v_j$ have degree $1$.

Suppose first that $\deg(u_2)=0$. Then we may assume  $u_2=1$. 
We can solve the system $g_1-\omega^jg_2=u_jv_j$ for $j=0,1$ and we can express $g_1$ and $g_2$ in terms of $u_0,v_0,u_1,v_1$.
In this case we have  $v_2=g_1-\omega^2 g_2$ is polynomial in $u_0,u_1,v_0,v_1$. We can then solve for $f_1,f_2$, and find the two toric decompositions. 
Let $P$ correspond to $(f_1,g_1)$ and $Q$ to $(f_2,g_2)$. Set $a=\langle P,Q\rangle$ and $b=\langle P,\omega Q\rangle$. Then using that $\omega$ is an automorphism we obtain the following Gram matrix
\[ \left(\begin{matrix} 2&-1& a & b \\
    -1 &2 &-a-b & a\\
    a &-a-b &2&-1\\
    b& a& -1 &2\\
   \end{matrix}\right)\]
Note that $a=1-\deg(u_0)$ and $b=1-\deg(u_1)$. 
Since $\deg(u_0)+\deg(u_1)+\deg(u_2)=3$ and $\deg(u_i)\leq 2$ we obtain $\{a,b\}=\{0,-1\}$. The corresponding lattice is a rank 4 root lattice with discriminant 4. There is exactly one such lattice, namely $D_4$. If $C$ has eight cusps then the Alexander polynomial is $(t^2-t+1)^2$. Hence the Mordell-Weil group has rank 4. So $D_4$ is a lattice of finite index in  the Mordell-Weil lattice. If this index would be bigger than one then the Mordell-Weil lattice would be a positive definite unimodular lattice of rank 4. However, such a lattice does not exists and hence we get that the Mordell-Weil lattice is $D_4$.

Consider now the second case, where $\deg(u_i)=\deg(v_j)=1$ for all $i,j$. Then without loss of generality we may assume  $u_0=x, u_1=y, v_0=z$ and $v_1=ax+bx+cz$.
We can solve the system $g_1-\omega^jg_2=u_jv_j$ for $j=0,1$. We know that   $g_1-\omega^2 g_2$ factors in two linear forms. A straightforward calculation shows that this happens if and only if $b=ac\omega$.
There are 6 points in $f_1=g_1=0$ and 6 in $f_2=g_2=0$, and all these points yield cusps of $C$. If a point is contained in both sets, i.e. $f_1=g_1=f_2=g_2=0$, then $xy=u_0v_0=g_1-g_2=0$. One can now easily find that there are three points on $f_1=f_2=g_1=g_2=0$. In particular, $C$ has at least 9 nine cusps, contradicting our assumption that we have 8 cusps.
\end{proof}
\begin{remark} This result combined with Degtyarev's proof of Oka's conjecture shows that if $C$ is a cuspidal sextic with 8 cusps then the height pairing is always $D_4$.
\end{remark}

\begin{proposition}
Let $C$ be a cuspidal sextic with $9$ cusps. Then $C$ has precisely 72 toric decompositions and the height pairing is $E_6$.
\end{proposition}
\begin{proof}
A sextic with 9 cusps is the dual of a smooth cubic. Hence there is only one family of such sextics. In particular, it suffices to determine the height pairing for one such curve. E.g. the curve 
\[x^6-2x^3y^3-2x^3z^3+y^6-2y^3z^3+z^6\]
has 9 cusps. Each of the conics $xy=0, yz=0,xz=0$ contains six cusps, which  yields in total 18 toric decomposition. The height pairing of the subgroup generated by these decomposition is $A_2\oplus A_2\oplus A_2$. This lattice has precisely 18 shortest vectors.

The determinant of this lattice is $9$. Hence either this is the Mordell-Weil lattice or the Mordell-Weil lattice is a root lattice of rank 6 and discriminant 3. There is only one such lattice, namely $E_6$. The six cusps \[(0:\omega:1),(0:1:1),(1:0:1),(\omega:0:1), (\omega^2:1:0)\mbox{ and }(\omega:1:0)\] are also on a conic, yielding a further 6 toric decompositions. Hence the Mordell-Weil lattice is $E_6$ and there are 72 shortest vectors.
\end{proof}
\begin{remark}

The Mordell-Weil lattices were first studied in the case of elliptic surfaces. Note that  we have $E(\C(s,t))\subset E(\overline{\C(s)}(t))$. In our case we have that the latter elliptic surface is a rational elliptic surface over the algebraically closed field $\overline{\C(s)}$ without reducible fibers. In this case we have that the Mordell-Weil lattices is $E_8$, and the Mordell-Weil group is generated by polynomials of degree at most 2 in $t$. The Mordell-Weil lattice associated with $E(\C(s,t))$ equals $E_8^G$, with $G=\Gal(\overline{\C(s)}/\C)$. Now $G$ acts through a subgroup of $W(E_8)$. Even though $E_8$ is generated by length 2 vectors, this statement is not true for every $E_8^H$ for an arbitrary subgroup $H$  of $W(E_8)$. However, Degtyarev's results now imply that $E_8^G$ is generated by length 2 vectors. It would be interesting to have a more group-theoretic or arithmetic proof of this fact.
(Note that $E$ is an elliptic curve with complex multiplication. This implies that the Galois action factors through a small subgroup of $W(E_8)$.)
\end{remark}

\section{Degree 12 cuspidal curves and Zariski pairs}\label{secZar}
The Mordell-Weil group of elliptic threefolds associated with cuspidal sextics are generated by toric decompositions. We will now show that this is not true for degree 12 curves. Moreover, we use our construction to exhibit a Zariski pair of degree 12 curves with 30 cusps, with the same Alexander polynomial.

We start by describing all cuspidal curves of degree $6k$ such that the Mordell-Weil lattice has a vector of length $2(k+1)$. 

For practical reasons we will use a slightly different notation from in the previous section. Let $C=V(F)$, with $F$ a homogeneous polynomial of degree $6k$.
We want to describe all $F$ which have a $(2,3,6)$-toric decomposition $(f,g,h)$ with $\deg(h)=1$.
Without loss of generality we may assume that $h=y_0$. Hence we  look for a polynomials $f$ and $g$ of degree $2(k+1)$ and $3(k+1)$, respectively, such that
\[ f^3-g^2=y_0^6F\]
holds. 
Write $f=\sum f_i(y_1,y_2)y_0^i$ and $g=\sum g_i(y_1,y_2)y_0^i$. An annoy\-ing but straight-forward computation then shows that  such a factorization exists if the $f_i,$ $g_i$ satisfy the relations from Table~\ref{tblRel}.
The $u,f_1',f_2'$ of this table are  forms in $y_1,y_2$ of the appropriate degree. If we assume that $C$ is irreducible then $u$ is nonzero and we find a reverse statement, i.e, the above equation are also necessary conditions.

\begin{table}[hbtp]
\begin{eqnarray*}
f_0&=&u^2\\
f_1&=&uf_1'\\
f_2&=&\frac{1}{4}(f_1')^2 + uf_2'\\
g_0&=&u^3\\
g_1&=&\frac{3}{2}u^2f_1'\\
g_2&=&\frac{3}{4}(u(f_1')^2 + 2u^2f_2')\\
g_3&=&\frac{1}{8}((f_1')^3 + 6uf_1'f_2' + 12f_3u)\\
g_4&=&\frac{3}{8} (u (f_2')^2 + 4 f_4u + 2f_3f_1')\\
g_5&=&\frac{3}{16}(-f_1'(f_2')^2 + 8f_5 u + 4f_4f_1' + 4f_3f_2').
\end{eqnarray*}
\caption{Relations among the $f_i$ and $g_j$}
\label{tblRel}
\end{table}

We claim that a general solution of the above system has $6k^2+4k-2$ cusps.
Indeed, we have that $f=0$ and $g=0$ intersect $6(k+1)^2$ points. 
If $y_0\neq 0$ at such a point then we have a cusp. If $y_0=0$ then also $u$ vanishes at this point. A straightforward computation shows that if we consider $u$ and $y_0$ as local coordinates and then the local intersection number of $f$ and $g$ (as polynomials in $\C(f_1',f_2',f_4,f_4,f_5)[u,y_0]$) equals $9$.
Hence the number of distinct intersetion points of $f=0$ and $g=0$ is $6(k+1)^2-8(k+1)$. Note that $F$ has also cusps at the intersection points of $f$ and $g$ where $u$ vanishes. In particular $C=V(F)$ has  at least $6(k+1)^2-8(k+1)=6k^2+4k-2$ cusps.

In the case $k=1$ we have 8 cusps and  it turns out that $(f,g)$ is a linear combination of two toric decompositions.

In the case $k=2$ we checked with a computer in an explicit example that for one choice of the parameters we get a degree 12 curve $C$ with 30 cusps and no other singularities, and that $\delta_{5/6}$ equals $1$. The corresponding Mordell-Weil lattice is of rank 2 and contains $A_2(3)$, with Gram matrix
\[\left( \begin{matrix} 6&-3\\-3&6 \end{matrix}\right)\]

Niels Lindner \cite{LindnerDiplom} constructed an example of a cuspidal curve $C'$ of degree 12 with 30 cusps and Alexander polynomial $t^2-t+1$. For this, he started with a sextic $C_0$ with 6 cusps, admitting a toric decomposition. He pulled back $C_0$ under a map $\Ps^2\to\Ps^2$ ramified above three inflectional tangents of $C_0$. Since the sextic is of torus type, then same holds for the pullback. Lindner showed that the Mordell-Weil lattice has rank 2 and that the Mordell-Weil group contains $A_2(2)$.

To show that we have a Zariski pair we need:
\begin{lemma} The inner product spaces $A_2(2)\otimes \Q$ and $A_2(3)\otimes \Q$ are not isomorphic.
\end{lemma}

\begin{proof}
Over $\Q$ we can diagonalize the pairing on $A_2(k)$. We have that $A_2(k)\otimes \Q$ is equivalent with a diagonal matrix with entries $2k,6k$. If we multiply a diognal entry  by a nonzero square then we get an isomorphic inner product space. Hence $A_2(2)\otimes \Q$ is equivalent to the diagonal matrix with entries $1,3$ and  $A_2(3)\otimes \Q$ with diagonal entries $6,2$.

If they were equivalent over $\Q$ then there would exist rational numbers $x_1,x_2$ such that $x_1^2+3x_2^2=2$. A possible solution to this equation has to be $3$-adically integral. If we reduce the equation modulo 3 we find $x_1^2\equiv 2 \bmod 3$ and this equation has no solution. 
 \end{proof}
Collecting everything we find that.

\begin{theorem} There exists a Zariski pair $(C,C')$ of degree 12 curves with 30 cusps and no further singularities, both having Alexander polynomial $t^2-t+1$.
\end{theorem} 
\begin{proof}
 By construction $C$ and $C'$ are curves of degree 12 with 30 cusps and Alexander polynomial $t^2-t+1$. If they do not form a Zariski pair then the Mordell-Weil lattices are the same by Proposition~\ref{prpInvariance}. However, both Mordell-Weil lattices have rank 2. The first Mordell-Weil lattices contains $A_2(2)$ and the second contains $A_2(3)$. Hence if the Mordell-Weil lattices where the same then $A_2(2)\otimes \Q$ and $A_2(3)\otimes \Q$ would be isomorphic as inner product spaces, which we excluded above.
\end{proof}

\section{Application to elliptic surfaces and to families of isotrivial abelian varieties}\label{secEll}
Our methods have an interesting application to the case of elliptic surfaces over non-algebraically closed fields:
\begin{proposition}
 Suppose that $f\in \C[s,t]$ is an irreducible polynomial. Then the Mordell-Weil groups $E_i(\C(s,t))$ of the elliptic curves $E_1:y^2=x^3+fx$ and $E_2:y^2=x^3+f^2$ are finite.
\end{proposition}
\begin{proof}
We start with the latter case. 
After replacing $t$ by a polynomial $h(s,t)$ we may assume that $g=f(s,h(s,t))$ has degree divisible by 3, say $3k$.
Since the Mordell-Weil group of $y^2=x^3+g^2$ contains $E(\C(s,t))$ it suffices to consider this elliptic curve.
Consider now the elliptic threefold $x^3+y^3+g(z_1/z_0,z_2/z_0)z_0^{3k}$. The Weierstrass equation of this threefold is isomorphic to $y^2=x^3+g^2$. 
Since $g(z_1/z_0,z_2/z_0)z_0^{3k}$ is an irreducible polynomial it follows from the main result of \cite{ZarIrr} that  the Alexander polynomial $\Delta_C$ of the associated curve does not vanish at $\zeta(m/p^a)$ for every $m$ coprime with $p$ and every integer $a$. In particular $\Delta_C$ does not vanish at $\zeta(1/3)$ and $\zeta(2/3)$. From our main theorem it then follows that the Mordell-Weil rank of $E(\C(s,t))$ is zero. 

In the former case we can argue similarly.
We first reduce to the case where the degree is divisible by $4$. Then we consider $y^2=x^4+f$. A direct computation shows that the Weierstrass equation of this elliptic curve is $y^2=x^3+fx$.
Then from Zariski's result it follows that the Alexander polynomial of $C$ does not vanish at $\zeta(1/4)$, $\zeta(1/2)$ and at $\zeta(3/4)$. In particular, from our main result it follows that the Mordell-Weil rank is zero.
\end{proof}

In particular we showed
\begin{proposition} Let  $L=\C(s)$ the function field of $\Ps^1$. If $f\in L[t]$ is an irreducible polynomial  then the rank of $E_i(L(t))$ for the elliptic curves
 \[ E_1:y^2=x^3+f(t)^2 \mbox{ and } E_2:y^2=x^3+f(t)x\]
is zero.
\end{proposition}
\begin{remark}
The above statement if also true if we take for $L$ an algebraically closed field. In that case the degree of $f$ is one and hence the elliptic surfaces over $L$ associated with $E_1$ and $E_2$ are rational elliptic surfaces with precisely two singular fibers.

If $L$ does not contain a third root of unity then $t^2+27$ is irreducible and the elliptic surface
\[ y^2=x^3+(t^2+27)^2\]
has a non-torsion $L(t)$-rational point with $x$-coordinate $\frac{t^2}{9}-9$. If $L$ does not contain a fourth root of unity then $t^2+1$ is irreducible and
\[ y^2=x^3-(t^2+1)x\]
has a non-torsion $L(t)$-rational point with $x$-coordinate $t^2+1$.

This suggest that the above proposition is true for very few non-algebra\-ical\-ly closed fields $L$.
\end{remark}

It turns out that in this case we can obtain a much stronger result if we use the Albanese variety instead of  Thom-Sebastiani:
\begin{theorem}
Let $A/\C(s,t)$ be an abelian variety (with  neutral element $O \in A(\C(s,t))$) such that there exists a finite cyclic extension $K/\C(s,t)$ of prime power degree $p^n$ satisfying the following properties:
\begin{enumerate}
\item there exists an abelian variety curve $A_0/\C$ with $A_K\cong (A_0)_K$; 
\item the ramification divisor $K/\C(s,t)$ is a prime divisor and
\item the group $A(\C(s,t))$ is finitely generated.
\end{enumerate}
Then $A(\C(s,t))$ is finite.
\end{theorem}
\begin{proof}
The abelian variety $A/\C(s,t)$ corresponds with a complex variety $\cA$ together to a rational map $\cA\dashrightarrow \Ps^2$ such that the general fiber is an abelian variety. Moreover, our assumptions imply that the general fiber is isomorphic to $A_0$.

Let $C\subset \Ps^2$ be the curve corresponding to the ramification divisor of $K/\C(s,t)$. Let $S$ be the ramified cover of degree $p^n$ of $\Ps^2$ ramified over $C$ and a multiple of a general line, if necessary.
Then by Zariski's result \cite{ZarIrr} we have that $h^1(S)=0$.

A $\C(s,t)$-point of $A$ yields a rational section $\Ps^2\dashrightarrow \cA$. We can pull this section back under the base change map and obtain a rational section $S \dashrightarrow S\times A_0$. Such a rational section is the graph of a rational map $S\dashrightarrow A_0$. Since a rational map into an abelian variety is defined everywhere \cite[Theorem 3.2]{milneAV} we obtain that the section is the graph of  a morphism $\varphi:S\to A_0$. Since $h^1(S)=0$ it follows that this morphism is constant and that the image of this map is a torsion point. 
Hence the corresponding point of $A(\C(s,t))$ is also a torsion point. Since $A(\C(s,t))$ is finitely generated this implies that $A(\C(s,t))$ is finite.
\end{proof}
\begin{remark}
A weaker form of this result was obtained by Libgober \cite[Theorem 1.2]{LibAlb}. 
\end{remark}

\begin{remark} 
 In the case of plane curves $H/\C(s,t)$ we did not have to assume that $J(H)(\C(s,t))$ is finitely generated. In that case we know that the corresponding threefold $X$ in a weighted projective space has $h^5(X)=0$. If $\tilde{X}$ is a resolution of singularities then $h^5(\tilde{X})$ vanishes also. Using Poincar\'e duality we get that $h^1(\tilde{X})=0$. Hence $\Pic(\tilde{X})$ is finitely generated. Since there is a surjective morphsim $\Pic(\tilde{X})\to J(H)(\C(s,t))$, also the latter group is also finitely generated.
\end{remark}

\begin{corollary}
Let $A,B\in \C$ such that $4A^3+27B^2\neq 0$. Let $f\in \C[s,t]$ be an irreducible polynomial. Then the Mordell-Weil ranks over $\C(s,t)$ of the elliptic curves
\[ y^2=x^3+Af^2x+Bf^3,\;y^2=x^3+f^3x\mbox{ and }y^2=x^3+f^4\]
are all zero.
\end{corollary}

\begin{remark}
Zariski's theorem \cite{ZarIrr} is crucial to our proves. For a modern proof see \cite{ArtDim}.
\end{remark}

\bibliographystyle{plain}
\bibliography{remke2}

\begin{thebibliography}{10}

\bibitem{ArtDim}
E.~Artal~Bartolo and A.~Dimca.
\newblock On fundamental groups of plane curve complements.
\newblock Preprint, available at \texttt{arXiv:1507.08178v1}, 2015.

\bibitem{CogLib}
J.-I. Cogolludo-Agust{\'{\i}}n and A.~Libgober.
\newblock Mordell-{W}eil groups of elliptic threefolds and the {A}lexander
  module of plane curves.
\newblock {\em J. Reine Angew. Math.}, 697:15--55, 2014.

\bibitem{DegOkaConj}
A.~Degtyarev.
\newblock Oka's conjecture on irreducible plane sextics.
\newblock {\em J. Lond. Math. Soc. (2)}, 78:329--351, 2008.

\bibitem{DegOkaConja}
A.~Degtyarev.
\newblock Oka's conjecture on irreducible plane sextics. {II}.
\newblock {\em J. Knot Theory Ramifications}, 18:1065--1080, 2009.

\bibitem{DimDif}
A.~Dimca.
\newblock Differential forms and hypersurface singularities.
\newblock In {\em Singularity theory and its applications, {P}art {I}
  ({C}oventry, 1988/1989)}, volume 1462 of {\em Lecture Notes in Math.}, pages
  122--153. Springer, Berlin, 1991.

\bibitem{Dim}
A.~Dimca.
\newblock {\em Singularities and topology of hypersurfaces}.
\newblock Universitext. Springer-Verlag, New York, 1992.

\bibitem{OkaConj}
C.~Eyral and M.~Oka.
\newblock On the fundamental groups of the complements of plane singular
  sextics.
\newblock {\em J. Math. Soc. Japan}, 57:37--54, 2005.

\bibitem{Kuli}
V.~S. Kulikov.
\newblock {\em Mixed {H}odge structures and singularities}, volume 132 of {\em
  Cambridge Tracts in Mathematics}.
\newblock Cambridge University Press, Cambridge, 1998.

\bibitem{LibPlane}
A.~Libgober.
\newblock Alexander invariants of plane algebraic curves.
\newblock In {\em Singularities, {P}art 2 ({A}rcata, {C}alif., 1981)},
  volume~40 of {\em Proc. Sympos. Pure Math.}, pages 135--143. Amer. Math.
  Soc., Providence, RI, 1983.

\bibitem{LibAdjHS}
A.~Libgober.
\newblock Position of singularities of hypersurfaces and the topology of their
  complements.
\newblock {\em J. Math. Sci.}, 82:3194--3210, 1996.
\newblock Algebraic geometry, 5.

\bibitem{LibIso}
A.~Libgober.
\newblock On {M}ordell-{W}eil groups of isotrivial abelian varieties over
  function fields.
\newblock {\em Math. Ann.}, 357:605--629, 2013.

\bibitem{LibAlb}
A.~Libgober.
\newblock Albanese varieties of abelian covers.
\newblock {\em J. Singul.}, 12:105--123, 2015.

\bibitem{LindnerDiplom}
N.~Lindner.
\newblock {C}uspidal plane curves of degree 12 and their {A}lexander
  polynomials.
\newblock Master's thesis, Humboldt Universit\"at zu Berlin, Berlin, 2012.

\bibitem{milneAV}
J.~S. Milne.
\newblock Abelian varieties (v2.00), 2008.
\newblock Available at www.jmilne.org/math/.

\bibitem{MiES}
R.~Miranda.
\newblock {\em The basic theory of elliptic surfaces}.
\newblock Dottorato di Ricerca in Matematica. ETS Editrice, Pisa, 1989.

\bibitem{PSbook}
C.~A.~M. Peters and J.~H.~M. Steenbrink.
\newblock {\em Mixed {H}odge structures}, volume~52 of {\em Ergebnisse der
  Mathematik und ihrer Grenzgebiete. 3. Folge.}
\newblock Springer-Verlag, Berlin, 2008.

\bibitem{SchSte}
J.~Scherk and J.~H.~M. Steenbrink.
\newblock On the mixed {H}odge structure on the cohomology of the {M}ilnor
  fibre.
\newblock {\em Math. Ann.}, 271:641--665, 1985.

\bibitem{ShiMW}
T.~Shioda.
\newblock On the {M}ordell-{W}eil lattices.
\newblock {\em Comment. Math. Univ. St. Paul.}, 39:211--240, 1990.

\bibitem{VarSC}
A.~N. Varchenko.
\newblock Semicontinuity of the spectrum and an upper bound for the number of
  singular points of the projective hypersurface.
\newblock {\em Dokl. Akad. Nauk SSSR}, 270:1294--1297, 1983.

\bibitem{ZarIrr}
O.~{Zariski}.
\newblock {On the linear connection index of the algebraic surfaces
  $z^n=f(x,y)$.}
\newblock {\em {Proc. Natl. Acad. Sci. USA}}, 15:494--501, 1929.

\end{thebibliography}

\end{document}